\documentclass[12pt]{amsart}

\usepackage{microtype, mathrsfs, xfrac, hyperref, amssymb, enumerate, xspace}
\usepackage{amsmath, amsthm, url}
\usepackage[latin1]{inputenc}

\numberwithin{equation}{section}

\numberwithin{subsection}{section}

\allowdisplaybreaks[1]

\newenvironment{enumerate1}
{\begin{enumerate}[\upshape (1)]}
{\end{enumerate}}

\newtheorem*{namedtheorem}{\theoremname}
\newcommand{\theoremname}{testing}

\newtheorem{theorem}{Theorem}[section]
\newtheorem{proposition}[theorem]{Proposition}
\newtheorem{proposition-definition}[theorem]
{Proposition-Definition}
\newtheorem{corollary}[theorem]{Corollary}
\newtheorem{lemma}[theorem]{Lemma}

\theoremstyle{definition}

\newtheorem{conjecture}[theorem]{Conjecture}
\newtheorem{example}[theorem]{Example}

\newtheorem{remark}[theorem]{Remark}

\theoremstyle{remark}

\renewcommand{\mathcal}{\mathscr}

 \newcommand\cB{\mathcal{B}}

\newcommand\cO{\mathcal{O}}

 \newcommand\cV{\mathcal{V}}
\newcommand\cW{\mathcal{W}} \newcommand\cX{\mathcal{X}}

\renewcommand\AA{\mathbb{A}} 
\newcommand\CC{\mathbb{C}} 
 \newcommand\FF{\mathbb{F}}

 \newcommand\PP{\mathbb{P}}
\newcommand\QQ{\mathbb{Q}} \newcommand\RR{\mathbb{R}}

 \newcommand\ZZ{\mathbb{Z}}

\newcommand\rC{\mathrm{C}}

 \newcommand\frm{\mathfrak{m}}

 \newcommand\frp{\mathfrak{p}}

\newcommand\arr{\ifinner\to\else\longrightarrow\fi}

\newcommand\arrto{\ifinner\mapsto\else\longmapsto\fi}

\newcommand{\xarr}{\xrightarrow}

\newcommand{\eqdef}{\mathrel{\smash{\overset{\mathrm{\scriptscriptstyle def}} =}}}

\def\displaytimes_#1{\mathrel{\mathop{\times}\limits_{#1}}}

\def\displayotimes_#1{\mathrel{\mathop{\bigotimes}\limits_{#1}}}

\newcommand\aut{\operatorname{Aut}}

\newcommand\spec{\operatorname{Spec}}

\newcommand\id{\mathrm{id}}

\newcommand\indlim{\varinjlim}

\newlength{\ignora}

\newcommand{\SL}{\mathrm{SL}}
\newcommand{\PGL}{\mathrm{PGL}}
\newcommand{\PSL}{\mathrm{PSL}}

\newcommand{\dr}[1]{(\mspace{-3mu}(#1)\mspace{-3mu})}
\newcommand{\ds}[1]{[\mspace{-2mu}[#1]\mspace{-2mu}]}

\DeclareFontFamily{U}{mathx}{\hyphenchar\font45}
\DeclareFontShape{U}{mathx}{m}{n}{
      <5> <6> <7> <8> <9> <10>
      <10.95> <12> <14.4> <17.28> <20.74> <24.88>
      mathx10
      }{}
\DeclareSymbolFont{mathx}{U}{mathx}{m}{n}
\DeclareFontSubstitution{U}{mathx}{m}{n}
\DeclareMathAccent{\widecheck}{0}{mathx}{"71}
\DeclareMathAccent{\wideparen}{0}{mathx}{"75}

\renewcommand{\epsilon}{\varepsilon}

\newcommand{\ed}{\operatorname{ed}}
\newcommand{\edloc}{\operatorname{ed}^{\rm loc}}

\newcommand{\cha}{\operatorname{char}}

\newcommand{\trdeg}{\operatorname{trdeg}}

\renewcommand{\geq}{\geqslant}

\begin{document}

\title[Essential dimension in mixed characteristic]{Essential dimension\\in mixed characteristic}

\author{Patrick Brosnan}
\address[Brosnan]{Department of Mathematics\\
1301 Mathematics Building\\
University of Maryland\\
College Park, MD 20742-4015\\
USA}
\email{pbrosnan@umd.edu}
\thanks{Patrick Brosnan was partially supported by 
NSF grant DMS 1361159.}

\author{Zinovy Reichstein}
\address[Reichstein]{Department of Mathematics\\
1984 Mathematics Road\\
University of British Columbia\\
Vancouver, BC V6T 1Z2\\Canada}
\email{reichst@math.ubc.ca}
\thanks{Zinovy Reichstein was partially supported by
National Sciences and Engineering Research Council of
Canada Discovery grant 253424-2017.}

\author{Angelo Vistoli}
\address[Vistoli]{Scuola Normale Superiore\\Piazza dei Cavalieri 7\\
56126 Pisa\\ Italy}
\email{angelo.vistoli@sns.it}
\thanks{Angelo Vistoli was partially supported by research funds from the Scuola Normale Superiore.}
\thanks{The authors are grateful to the Collaborative Research Group in Geometric and Cohomological Methods in Algebra at the Pacific Institute for the Mathematical Sciences for their support of this project.}



\begin{abstract}
  Let $G$ be a finite group, and let $R$ be a discrete valuation ring
  with residue field $k$ and fraction field $K$.  We say that $G$ is
  weakly tame at a prime $p$ if it has no non-trivial normal
  $p$-subgroups. By convention, every finite group is weakly tame at
  $0$.
  Using this definition, we show that
  if $G$ is weakly tame at $\cha(k)$, then
  $\ed_K(G) \geqslant \ed_k(G)$. Here $\ed_F(G)$ denotes the essential
  dimension of $G$ over the field $F$. We also prove a more general
  statement of this type, for a class of \'etale gerbes $\cX$ over
  $R$.

As a corollary, we show that if $G$ is weakly tame at $p$, then $\ed_{L} G  \geqslant \ed_{k} G$ for any field $L$ of characteristic $0$ 
and any field $k$ of characteristic $p$, provided that $k$ contains $\overline{\FF}_{p}$. We also show that a conjecture of A.~Ledet, asserting that
$\ed_k(\ZZ/p^n \ZZ) = n$ for a field $k$ of characteristic $p > 0$ implies that $\ed_{\CC}(G) \geqslant n$ for any finite group $G$ which 
is weakly tame at $p$ and contains an element of order $p^n$. We give a number of examples, where an unconditional proof of the last inequality is
out of the reach of all presently known techniques. 
\end{abstract}

\maketitle

\section{Introduction}

Let $R$ be a discrete valuation ring with residue field $k$ and fraction field $K$, 
and let $G$ be a finite group. In this paper we will compare $\ed_K(G)$ and $\ed_k(G)$.
More generally, we will compare $\ed_K(\cX)$ to $\ed_k(\cX)$ for an \'etale gerbe 
$\cX$ over $R$. For an overview of the theory of essential dimension, we refer the reader
to~\cite{brosnan-reichstein-vistoli3, merkurjev-survey, reichstein-icm}.

To state our main result, we will need some definitions.

Suppose $S$ is a scheme. By an \emph{\'etale gerbe} $\cX \arr S$ we mean an algebraic 
stack that is a gerbe in the \'etale topology on $S$. Furthermore, we will 
always assume that there exists an \'etale covering $\{S_{i} \arr S\}$, 
such that the pullback $\cX_{S_{i}}$ is of the form $\cB_{S_{i}}G_{i}$, 
where $G_{i} \arr S_{i}$ is a finite \'etale group scheme. 

We say that a finite group $G$ is \emph{tame} (resp. \emph{weakly tame}) at a prime number $p$ if $p\nmid |G|$ (resp. 
$G$ contains no non-trivial normal $p$-subgroup).
Equivalently, $G$ is tame at $p$ if the trivial group is the (unique) $p$-Sylow subgroup of $G$, and $G$ is weakly
tame at $p$ if the intersection of all $p$-Sylow
subgroups of $G$ is trivial.
By convention we say that every finite group is both tame and weakly tame at $0$.
\footnote{By a theorem of T.~Nakayama~\cite{nakayama}, $G$ is weakly tame at 
$p$ if and only if $G$ admits a faithful completely reducible 
representation over some (and thus every) field of characteristic $p$.
The significance of this condition in the study of essential dimension of 
finite groups was first observed by R.~L\"otscher~\cite{lotscher}. Note 
that L\"otscher used the term ``semifaithful" in place of ``weakly tame".}

By a geometric point of $S$, we mean a morphism $\spec\Omega \arr S$ with $\Omega$ an algebraically closed field.    We say that a finite \'etale
group scheme $G$ over $S$ is \emph{tame} (resp. \emph{weakly tame}) if, for every geometric point $\spec\Omega\to S$, the group $G(\Omega)$ is 
tame (resp. weakly tame) at $\cha\Omega$.  Similarly, we say that an \'etale gerbe $\cX \arr S$ is \emph{tame} (resp. \emph{weakly tame}) if, for every object $\xi$ over a geometric point $\spec\Omega \arr S$,
the automorphism group 
$\aut_{\Omega}\xi$ is tame (resp. weakly tame) at $\cha\Omega$.

A key result of \cite{brosnan-reichstein-vistoli3} is the so 
called \emph{Genericity Theorem} for tame Deligne--Mumford stacks, \cite[Theorem 6.1]{brosnan-reichstein-vistoli3}.  
The proof of this result in~\cite{brosnan-reichstein-vistoli3} was based on the following.   

\begin{theorem}[\hbox{\cite[Theorem~5.11]{brosnan-reichstein-vistoli3}}]\label{thmbrv3}
Let $R$ be a discrete valuation ring (DVR) with residue field $k$ and fraction field $K$, and let
\[ \cX\arr \spec R \]
be a tame \'etale gerbe. Then
$\ed_{K}\cX_K  \geqslant \ed_{k}\cX_{k}$.
\end{theorem}

Here $\cX_{K}$ and $\cX_{k}$ are respectively the generic fiber and the special fiber of $\cX \arr \spec R$.

Unfortunately, the proof of \cite[Theorem~5.11]{brosnan-reichstein-vistoli3} contains an error in the case when $\cha K = 0$ and $\cha k > 0$.  
This was noticed by Amit Hogadi, to whom we are very grateful. 
(See Remark~\ref{remex} for an explanation of the error.)
For the applications 
in \cite{brosnan-reichstein-vistoli3} only the equicharacteristic case 
was needed, so this mistake in the proof of Theorem~\ref{thmbrv3} does not affect any other results in~\cite{brosnan-reichstein-vistoli3} (the genericity theorem, in particular).
However, the assertion of Theorem~\ref{thmbrv3} in the mixed characteristic case remained of interest to us, as a way of
relating essential dimension in positive characteristic to essential dimension in characteristic $0$. 
In this paper, our main result  is the following strengthened version of Theorem~\ref{thmbrv3}.

\begin{theorem} \label{thm.main}
Let $R$ be a DVR with residue field $k$ and fraction field $K$, and let 
\[ \cX\arr \spec R \]
be a weakly tame \'etale gerbe. Then
$\ed_{K}\cX_K  \geqslant \ed_{k}\cX_{k}$.
\end{theorem}

In particular,~\cite[Theorem~5.11]{brosnan-reichstein-vistoli3} is valid as stated. 
Moreover, our new proof is considerably shorter than the one 
in~\cite{brosnan-reichstein-vistoli3}.
And in Sections~\ref{sect.cors}-\ref{sect.psl2} we will deduce some rather surprising consequences.

We will give two proofs of our main result, one for gerbes of the form 
where $\cX = \cB_{R}G$, where $G$ is a (constant) finite group
(Theorem~\ref{thm.main0}) and the other for the general case. 
The ideas in these two proofs are closely related; 
the proof of Theorem~\ref{thm.main0} allows us to
introduce these ideas in  
the elementary setting of classical valuation theory. 
A separate proof of Theorem~\ref{thm.main0} also 
makes the applications in Sections~\ref{sect.cors}-\ref{sect.psl2} accessible to those readers who are not 
familiar with, or don't care for, the language of gerbes.

\subsection*{Acknowledgements} 
We are grateful to the referee for a thorough reading and constructive suggestions. We would also like to thank
Alexander Duncan and Najmuddin Fakhruddin for helpful comments on an earlier version of this paper.

\section{Proof of Theorem~\ref{thm.main} in the constant case}
\label{sect.const}

In this section we will prove a special case of Theorem~\ref{thm.main}, 
where $\cX = \cB_{R}G$
for $G$ a finite group (viewed as a constant group scheme over $\spec R$); see Theorem~\ref{thm.main0}.

Throughout this section we will assume that $L$ is a field equipped with a (surjective) discrete valuation $\nu \colon L^* \to \ZZ$ and 
$K$ is a subfield of $L$ such that $\nu(K^*) = \ZZ$. We will denote the residue fields of $L$ and $K$ by $l$ and $k$, respectively.  Similarly, we will denote the valuation rings by $\mathcal{O}_L$ and $\mathcal{O}_K$.

The following lemma is a special case of the Corollary to Theorem 1.20 in~\cite{vaquie}.  For the convenience of the reader, we supply a short proof.

\begin{lemma} \label{lem.valuation} $\trdeg_k(l) \leqslant \trdeg_K(L)$.
\end{lemma}

\begin{proof} 
Let $u_{1}, \ldots, u_{m} \in l$ be algebraically independent over $k$. Lift each $u_i$ to $v_i \in \cO_{L} \subseteq L$. It now suffices to show that $v_1, \dots, v_m$ are algebraically independent over $K$. Assume the contrary: 
$f(v_{1}, \dots , v_{m}) = 0$ for some polynomial $0 \neq f(x_1, \dots, x_m) \in K[x_{1}, \dots, x_{m}]$. After clearing denominators 
we may assume that every coefficient of $f$ lies in 
$\cO_{K}$, and at least one of the coefficients has valuation~$0$. 
If $f_{0}$ is the image of $f$ in $k[x_{1}, \dots, x_{m}]$ then 
$f_{0} \neq 0$ and $f_{0}(u_{1}, \dots, u_{m}) = 0$. This contradicts 
our assumption that $u_1, \dots, u_m$ are algebraically independent over $k$.
\end{proof}

Let $L_m = \nu^{-1}(m) \cup \{ 0 \}$ and $L_{\geqslant m} = \bigcup_{j \geqslant m} \, L_j$. Note that, by definition,
$L_{\geqslant 0} = \cO_L$ is the valuation ring of $\nu$, 
$L_{\geqslant 1}$ is the maximal ideal, 
and $L_{\geqslant 0}/L_{\geqslant 1} = l$
is the residue field.

\begin{lemma} \label{lem.bur5.1}
Assume that $g$ is an automorphism of $L$ of finite order $d \geqslant 1$, preserving the valuation $\nu$. Let $p = \cha(l) \geqslant 0$.
If $g$ induces a trivial automorphism on both $L_{\geqslant 0}/L_{\geqslant 1}$ and $L_{\geqslant 1}/L_{\geqslant 2}$, then

\smallskip
(a) $d = 1$ (i.e., $g = \id$ is the identity automorphism) if $p = 0$, and

\smallskip
(b) $d$ is a power of $p$, if $p > 0$.
\end{lemma}

Part (a) is proved in~\cite[Lemma 5.1]{bur}; a minor variant of
the same argument
also proves (b).  Alternatively, with some additional work, 
Lemma~\ref{lem.bur5.1} can be deduced from~\cite[Theorem 25, p. 295]{zs1}.  
For the reader's convenience we will give a short self-contained 
proof below.

\begin{proof} In case (b), write $d = mp^{r}$, where $m$ is not divisible 
by $p$. After replacing $g$ by $g^{p^{r}}$, we may assume that $d$ is prime 
to $p$. In both parts we need to conclude that $g$ is the identity.

Let $G$ be the cyclic group generated by $g$; then $G$ 
is linearly reductive. Since the action of $G$ on $l$  is trivial, 
the induced action on $L_{\geq i}/L_{\geq i+1}$ is $l$-linear. 
Furthermore, let $t \in L_1$ be a uniformizing parameter. By our assumption
$g(t) = t \pmod{L_{\geqslant 2}}$. Thus multiplication by $t^{i-1}$ induces
the $l$-linear $G$-equivariant 
isomorphism $(L_{\geqslant 1}/L_{\geqslant 2})^{\otimes i} 
\simeq L_{\geq i}/L_{\geq i+1}$. Consequently, $G$ acts trivially 
on $L_{\geq i}/L_{\geq i+1}$ for all $i \geq 0$. Since $G$ is linearly 
reductive, from the exact sequence
   \[
   0 \arr L_{\geq i}/L_{\geq i+1} \arr L_{\geq 0}/L_{\geq i+1} \arr L_{\geq 0}/L_{\geq i} \arr 0
   \]
we deduce, by induction on $i$, that $G$ acts trivially on 
$L_{\geq 0}/L_{\geq i}$ for every $i \geqslant 1$. Since $\bigcap_{i \geq 0} L_{i} = 0$, this implies that the action of $G$ on $L_{\geq 0}$ is trivial. But $L_{\geq 0}$ is a domain with quotient field $L$, so $G$ also acts trivially on $L$. Since $G$ acts 
faithfully on $L$, we conclude that $G = \{1\}$, and the lemma follows.
\end{proof}

\begin{proposition} \label{prop.faithful} Consider a faithful action of a finite group $G$ on $L$, such that
$G$ preserves $\nu$ and acts trivially on $K$.  Let $\Delta$ be the kernel of the induced 
$G$-action on $l$. Then $\Delta = \{ 1 \}$ if $\cha(k) = 0$ and $\Delta$ is a $p$-subgroup if $\cha(k) = p$.
\end{proposition}

\begin{proof} Assume the contrary. Then we can choose an element $g \in \Delta$ of prime order $q$, such that $q \neq \cha (k)$.
Let $M$ be the maximal ideal of the valuation ring $\cO_L$. Since we are assuming that $\nu(K^*) = \nu(L^*) = \ZZ$, we can choose 
a uniformizing parameter $t \in K$ for $\nu$. 
Since $g \in \Delta$, $g$ acts trivially on both $l = \cO_L/M$ and $M/M^2 = l \cdot t$. By Lemma~\ref{lem.bur5.1}, 
$g$ acts trivially on $L$. This contradicts our assumption that $G$ acts faithfully on $L$.
\end{proof}

We are now ready to prove the main result of this section.

\begin{theorem} \label{thm.main0}
Let $(R, \nu)$ be a discrete valuation ring with residue field $k$ and fraction field $K$, and
$G$ be a finite group. If $p = \cha(k) > 0$, assume that $G$ is weakly tame at $p$.
Then $\ed_{K}(G)  \geqslant \ed_{k}(G)$.
\end{theorem}

\begin{proof} Set $d \eqdef \ed_K(G)$. Let $R[G]$ be the group algebra of $G$ and let $V_R = (\mathbb{A}_R)^{|G|}$ denote
the corresponding $R$-scheme equipped with the (left) regular action of $G$.
By definition $d$ is the minimal transcendence degree
$\trdeg_K(L)$, where $L$ ranges over $G$-invariant intermediate subfields $K \subset L \subset K(V_K)$ such that the $G$-action on $L$ is faithful; see~\cite{bur}. Choose a $G$-invariant intermediate subfield $L$ such that $\trdeg_K(L) = d$. 

We will now construct a $G$-invariant intermediate 
subfield $k \subset l \subset k(V_k)$, where $V_k$ is the regular 
representation of $G$ over $k$, as follows.
Lift the given valuation $\nu \colon K^* \to \ZZ$ to the purely transcendental extension $K(V_K)$ of $K$ in the obvious way.
That is, $\nu: K(V_K)^* \to \mathbb{Z}$ is the divisorial valuation corresponding to the fiber of $V_R$ over the closed point in $\spec R$.
The residue field of $K(V_K)$ is then $k(V_k)$. 
By restriction, $\nu$ is a valuation on $L$ with $\nu(L^*)=\mathbb{Z}$.  
Let $l$ be the residue field of $L$. Clearly $k \subset l \subset k(V_k)$ and $\nu$ is
invariant under $G$. By Proposition~\ref{prop.faithful},
$G$ acts faithfully on $l$. Moreover, by Lemma~\ref{lem.valuation}, $\trdeg_k(l) \leqslant d$. Thus
$\ed_k(G) \leqslant d = \ed_K(G)$, as desired.
\end{proof}

\section{Examples illustrating Theorem~\ref{thm.main0} and a simple application}
\label{sect.cors}

\begin{example} \label{ex.necessary1} The following example shows 
that Theorem~\ref{thm.main0} fails if we do not assume that
$G$ is weakly tame. Choose $R$ so that
$\cha K = 0$, $\cha k = p > 0$, and $K$ contains a $p^{2}$-th root of $1$. 
Let $G = C_{p^2}$ be the cyclic group of order $p^{2}$. 
Since $K$ contains a primitive $p^2$-th root of $1$,
$\ed_K (G) = \ed_{K} (\rC_{p^{2}}) = 1$. On the other hand, 
$\ed_k(G) = \ed_{k} (\rC_{p^{2}}) = 2$; this is a special (known) case of Ledet's conjecture, 
see~Remark~\ref{rem.ledet}.
\end{example}

\begin{example} \label{ex.necessary2}
Here is an example showing that Theorem~\ref{thm.main0} fails if we do not 
assume that $R$ is a DVR. Let $R \subseteq \CC\ds{t}$ be the subring 
consisting of power series in $t$ 
whose constant term is real. Then $R$ is a one-dimensional complete 
Noetherian local ring with quotient field $K = \CC\dr{t}$ and 
residue field $k = \RR$, but not a DVR. Letting $G = C_4$ be the 
cyclic group of order $4$, we see that in this situation 
$\ed_K(G) = \ed_{\CC\dr{t}} (\rC_{4}) = 1$, 
while $\ed_k(G) = \ed_{\RR} (\rC_{4}) = 2$; see \cite[Theorem 7.6]{bf1}.
\end{example}

\begin{example} \label{ex.semicontinuity} (cf.~\cite[Remark 4.5(ii)]{tossici})
This example shows that essential dimension is not semicontinuous 
in any reasonable sense, even in characteristic $0$. Consider the scheme
   \[
   S \eqdef \spec\QQ[u, x]/(x^{2} - u)\,.
   \]
The embedding $\QQ[u] \subseteq \QQ[u, x]/(x^{2} - u)$ gives a finite map $S \arr \AA^{1}_{\QQ}$. If $p$ is an odd prime, the inverse image of the prime $(u - p) \subseteq \QQ[u]$ in $S$ consists of a point $s_{p}$ with residue field $k(s_{p}) =\mathbb{Q}(\sqrt{p}) = \QQ[x]/(x^{2} - p)$. Then $\ed_{\QQ(\sqrt{p})}(\rC_{4}) = 1$ if $-1$ is a square modulo $p$, and $\ed_{\QQ (\sqrt{p})}\rC_{4} = 2$ if $-1$ is not a square modulo $p$; once again, see~\cite[Theorem 7.6]{bf1}. Equivalently,
$\ed_{\QQ(\sqrt{p})}(\rC_{4}) = 1$ if $p \equiv 1 \pmod{4}$, and $\ed_{\mathbb{Q}(\sqrt{p})}\rC_{4} = 2$ is $p  \equiv 3 \pmod{4}$. 
We conclude that the set of points $s \in S$ with $\ed_{k(s)}\rC_{4} = 1$ is dense in $S$, and likewise for the set of points $s \in S$ with $\ed_{k(s)}\rC_{4} = 2$
is also dense in $S$. 
\end{example}

We conclude this section with an easy corollary 
of Theorem~\ref{thm.main0}. 

\begin{corollary} \label{cor1}
Let $p$ be a prime, $G$ a finite group weakly tame at $p$. Then 
\begin{enumerate1}

\item[(a)] {\rm(}cf.~\cite[Corollary 4.2]{tossici}{\rm)}  $\ed_{\QQ}G \geqslant\ed_{\FF_{p}}G$.

\item[(b)] If $K$ is a field of characteristic $0$ and $k$ a field of characteristic $p$ containing $\overline{\FF}_{p}$, then $\ed_{K}G \geqslant \ed_{k}G$.
\end{enumerate1}
\end{corollary}

\begin{proof}
(a) follows directly from Theorem~\ref{thm.main0} by taking $R$ to be the localization of the ring of integers $\ZZ$ at a prime ideal $p \ZZ$.

(b) Let $\overline{K}$ be the algebraic closure of $K$. Since 
$\ed_K(G) \geq \ed_{\overline{K}}(G)$, we may replace that $K$ 
by $\overline{K}$ and thus assume that $K$ is algebraically closed. 
Note that $\ed_{K}G = \ed_{\overline{\QQ}}G$ and $\ed_{k}G = \ed_{\overline{\FF}_{p}}G$; 
see~\cite[Proposition 2.14]{brv1a} or~\cite[Example 4.10]{tossici}.

Choose a number field $E \subseteq \overline{\QQ}$
such that $\ed_{E}G = \ed_{\overline{\QQ}}G$
and let $\frp \subseteq \cO_{E}$
a prime in the ring $\cO_{E}$
of algebraic integers in $E$
lying over $p$.
Set $E_{0} \eqdef \cO_{E}/\frp$.
Since $k$
contains $\overline{\FF}_{p}$,
there is an embedding $E_{0} \subseteq k$.
By Theorem~\ref{thm.main0}, $\ed_{E}(G) \geqslant \ed_{E_{0}}(G)$
and since $E_0 \subset k$, $\ed_{E_{0}}G \geqslant \ed_{k}G$.
\end{proof}

\begin{example} \label{ex.equality}
A.~Duncan pointed out to us that equality in Corollary~\ref{cor1}(b) does not always hold.
For example, let $G = A_5$ be the alternating group of order $60$ and
$p = 2$. Note that since $A_5$ is simple, it is weakly tame at every prime.
By \cite[Theorem 6.7]{bur}, $\ed_{\CC}(A_5) = 2$.
On the other hand, 
$A_5 \simeq \SL_2(\mathbb{F}_4)$ admits a $2$-dimensional faithful linear representation 
over any field $k$ containing $\mathbb{F}_4$, that is, the representation coming from the obvious  inclusion of $\SL_2(\FF_4)$ into $\SL_2(k)$. The natural ($A_5$-equivariant) 
projection $\AA^2 \dasharrow \PP^1$ now tells us that $\ed_k(A_5) = 1$. In summary,
\[ 2 = \ed_{\CC}(A_5) > \ed_k(A_5) = 1 . \]
\end{example}

\begin{remark} \label{rem.inequality}
The group $G = A_5$ in Example~\ref{ex.equality} is weakly tame but not tame at $2$. 
We do not know of any such examples with $G$ tame.
We conjecture that they do not exist. That is, if $|G|$ is prime to $p$, then under the hypotheses 
of Corollary~\ref{cor1}(b), $\ed_{K}G =  \ed_{k}G$, provided that 
$K$ is algebraically closed.
\end{remark}

\section{Ledet's conjecture and its consequences}

The following conjecture is due to A.~Ledet \cite{ledet-p}.

\begin{conjecture}\label{conj.ledet}
If $k$ is a field of characteristic $p > 0$, $n$ is a natural number, and $\rC_{p^{n}}$ is a cyclic group of order $p^{n}$, then $\ed_{k}(\rC_{p^{n}}) = n$.
\end{conjecture}

\begin{remark} \label{rem.ledet}
It is known that in characteristic $p$, $\ed(\rC_{p^n}) \leqslant n$ for every $n \geq 1$ (see~\cite{ledet-p}) and $\ed(\rC_{p^n}) \geqslant 2$ 
if $n \geqslant 2$ (\cite[Theorems 5 and 7]{ledet-ed1}). Thus the conjecture 
holds for $n = 1$ and $n = 2$; it remains open 
for every $n \geqslant 3$.
\end{remark}

Combining Conjecture~\ref{conj.ledet} with Theorem~\ref{thm.main0}, we obtain the following surprising result. 

\begin{proposition} \label{prop2}
Assume that a finite group $G$ is weakly tame at a prime $p$ and contains an element of order $p^n$. Let $K$ be a field of characteristic $0$. 
If Conjecture~\ref{conj.ledet} holds for $C_{p^n}$, then $\ed_{K}(G) \geqslant n$.
\end{proposition}

\begin{proof} By Corollary~\ref{cor1}(b), with $k = \overline{\FF}_p$, we have $\ed_K(G) \geqslant \ed_k(G)$. Since $G$ contains $\rC_{p^n}$, $\ed_k(G) \geqslant \ed_k(\rC_{p^n})$, and by Conjecture~\ref{conj.ledet}, $\ed_k(\rC_{p^n}) = n$.
\end{proof}

%

\begin{corollary} \label{cor.semidirect}
Let $p$ be a prime and $n$ a positive integer. Choose a positive integer 
$m$ such that $q \eqdef mp^{n}+1$ is a prime. 
(By Dirichlet's theorem on primes in
arithmetic progressions, there are infinitely many such $m$.) 
Let $\rC_{q}$ be a cyclic group of order $q$. Then $\aut\rC_{q} = (\ZZ/q\ZZ)^{*}$ is cyclic of order $mp^{n}$; let $\rC_{p^{n}} \subseteq (\ZZ/q\ZZ)^{*}$ denote the subgroup of order $p^{n}$. Set 
$G \eqdef \rC_{p^{n}} \ltimes \rC_{q}$. Then

\smallskip
(a) $G$ is weakly tame at $p$, and

\smallskip
(b) if Conjecture~\ref{conj.ledet} holds, then $\ed_K(G) \geqslant n$ for any field $K$ of characteristic $0$. 
\end{corollary}

\begin{proof}
(a) Suppose $S \subseteq G$ is a normal $p$-subgroup. Then $S$ lies in every Sylow $p$-subgroup of $G$, in particular, in $C_{p^n}$. Our goal is to show that $S = \{ 1 \}$.
The cyclic group $C_q$ of prime order $q$
acts on $S$ by conjugation. Since $q > p^n \geqslant |S|$, this action is trivial. In other words,
$S$ is a central subgroup of $G$. In particular, $S$ acts trivially on $C_q$ by conjugation.  On the other hand, by the definition of $G$,
$C_{p^n}$ acts faithfully on $C_{q}$ by conjugation. We conclude that $S = \{1\}$, as desired.

(b) follows from Proposition~\ref{prop2}.
\end{proof}

\begin{remark} \label{rem.semidirect1} The inequality of Corollary~\ref{cor.semidirect}(b) is equivalent to 
\begin{equation} \label{e.semidirect}
\ed_\CC (\rC_{p^n} \ltimes \rC_{q}) \geqslant n \, ,
\end{equation}
where $\CC$ is the field of complex numbers (once again, see~\cite[Proposition 2.14]{brv1a} or~\cite[Example 4.10]{tossici}). For $n = 2$ and $3$,
this inequality can be proved unconditionally 
(i.e., without assuming Conjecture~\ref{conj.ledet})
by appealing to the classifications of finite groups of essential 
dimension $1$ and $2$ over $\CC$ in~\cite[Theorem 6.2]{bur} and \cite[Theorem 1.1]{duncan-ed2} respectively.
\end{remark}

\begin{remark} \label{rem.semidirect2} Let $G$ be a finite group. Set
\[ \edloc_k(G) := \max  \, \{ \ed(G; p) \, | \, \text{$p$ is a prime} \}, \]
where $\ed_k(G; p)$ denotes essential dimension of $G$ at a prime $p$ and the superscript
``loc" stands for ``local". If the base field $k$ is assumed to be fixed, we will write $\ed(G; p)$ and $\edloc(G)$
in place of $\ed_k(G; p)$ and $\edloc_k$, respectively.

Clearly $\ed(G) \geqslant \edloc(G)$. 
In the language of \cite[Section 5]{reichstein-icm}, computing 
$\edloc(G)$ is a Type I problem. This problem is solved, at least in principle, by the Karpenko-Merkurjev theorem~\cite{km2}.
Computing $\ed(G)$ in those case, where $\ed(G) > \edloc(G)$ is a Type II problem. Such problems tend to be very hard. For more on this, see \cite[Section 5]{reichstein-icm} or the discussion after the statement of Theorem 2 in~\cite{reichstein-cremona}.

Let us now return to the setting of Corollary~\ref{cor.semidirect}, where $G = \rC_{p^n} \ltimes \rC_{q}$.
Since all Sylow subgroups of $G = \rC_{p^n} \ltimes \rC_{q}$ are cyclic, one readily sees that
$\edloc_{\CC}(G) = 1$.
Thus the inequality~\eqref{e.semidirect} is a ``Type 2 problem"
whenever $n \geqslant 2$.
An unconditional proof of this inequality is out of the reach of all
currently available techniques for any $n\geqslant 3$. However, it is shown
in~\cite{reichstein-cremona} that
\[ \lim_{n \to \infty} \ed_{\CC} (\rC_{p^{n}} \ltimes \rC_{q}) \longrightarrow \infty \]
for any choice of $q$.
\end{remark}

\begin{remark} \label{rem.ed-p} It is shown in~\cite{rei-vi} that if $G$ is a finite group and $k$ is a field of characteristic $p$,
then
\begin{equation} \label{e.ed-p}
\ed_k(G; p) = \begin{cases} \text{$1$, if $p$ divides $|G|$, and} \\
\text{$0$, otherwise.}
\end{cases}
\end{equation}
In particular, $\edloc_k(C_{p^n}) = 1$ for every $n \geqslant 1$. So, for $n \geqslant 2$,
Conjecture~\ref{conj.ledet} is also a Type 2 problem. Thus the situation in Corollary~\ref{cor.semidirect} can be described as follows:
we deduce one Type II assertion from another, without being able to prove either one from first principles. Another results of this type is \cite[Proposition 10.8]{duncan-reichstein}; further examples can be found in the next section.
\end{remark}

\begin{remark} \label{rem.main0-at-p}
In view of~\eqref{e.ed-p}, Corollary~\ref{cor1}(b)
continues to hold if we replace essential dimension by essential dimension at $p$, for trivial reasons. Moreover,
under the assumptions of Corollary~\ref{cor1},
(a$^{\prime}$) $\ed_{\QQ}(G; p) \geqslant\ed_{\FF_{p}}(G; p)$ 
and (b$^{\prime}$) $\ed_{K}(G; p) \geqslant \ed_{k}(G; p)$, for any
finite group $G$, not necessarily weakly tame. In (b$^{\prime}$) we can also drop the requirement that $k$ should contain $\overline{\FF}_p$.
Note however that our proof of Theorem~\ref{thm.main0} breaks down 
if we replace essential dimension by essential dimension at $p$. 
\end{remark}

\section{Essential dimension of $\PSL_2(q)$}
\label{sect.psl2}

Let $p$ be a prime, $q = p^r$ be a prime power and $\mathbb F_q$ be a
field of $q$ elements.  Let $G = \PSL_2(q) = \PSL(2, \FF_q)$.  (To
avoid confusion, we remind the reader that $G$ is the quotient of
$\SL(2,\FF_q)$ by its subgroup $\{\pm 1\}$.  In general, it is not the
same thing as the group $\PSL_2(\FF_q)$ of $\FF_q$ points of the
algebraic group $\PSL_2 = \PGL_{2}$.)   For $q>3$, it is well-known
that $G$ is simple; see, e.g.,~\cite[p.~39]{Dieudonne}
or~\cite[p.~419]{gfg}.  Hence, $G$ is weakly tame at every prime. In
this section we will work over the field $k = \CC$ of complex numbers
and deduce lower bounds on $\ed_{\CC}(G)$ from Ledet's conjecture.

For some $q$, these lower bounds are Type II bounds, in the sense of
Remark~\ref{rem.semidirect2}, and are genuinely new. To establish this
we will compute $\edloc(G)$ in every case. We begin with the following
well-known description of the Sylow subgroups of $\PSL_2(q)$.

\begin{lemma}\label{dih}
Let $p$ and $\ell$ be prime numbers and set $q=p^r$ for some positive
integer $r$.  Let $G_{\ell}$ denote an $\ell$-Sylow subgroup
of $G=\PSL_2(q)$.  Then 
\begin{enumerate}
    \item[(a)] For $\ell=p$, we have $G_{\ell}\cong (C_p)^r$.
    \item[(b)]  For $\ell\not\in\{2,p\}$, $G_{\ell}$ is cyclic.
    \item[(c)] For $p$ odd and $\ell=2$, $G_{\ell}$ is 
    dihedral. 
\end{enumerate}
\end{lemma}
\begin{proof} See~\cite[Lemma 1.1 on page 418]{gfg}.
\end{proof}   

\begin{proposition} \label{prop.psl2} 
Let $p$ be a prime and $q = p^r$ be a prime power.

\smallskip
(a) $\edloc \PSL_2(q)= \begin{cases} r,  & \text{$q$ even};\\
 \max(2,r), & \text{$q$ odd}. 
\end{cases} $

\smallskip
(b)  Let $\ell$ be a prime and $s$ be a nonnegative
integer such that 
$2 \ell^s$ divides $q^2 - 1$. 
If Ledet's Conjecture~{\rm \ref{conj.ledet}} holds for cyclic groups of
order $\ell^s$ in characteristic $\ell$, then 
$\ed_{\CC} ( \PSL_2(q) ) \geqslant s$. 
\end{proposition}

Note that part (b) is conditional on Ledet's conjecture but part (a) is not.  

\begin{proof} 
  Set $G=\PSL_2(q)$.  We begin by pointing out that
 \begin{equation}\label{Gsize}
        |G|=\begin{cases}
        (q-1)q(q+1)/2, & 2\nmid q;\\
        (q-1)q(q+1),   & 2| q.
   \end{cases}
    \end{equation}

(a) Recall that $\ed_{\CC}(G;\ell) = \ed_{\CC}(G_{\ell}; \ell)$, where $G_{\ell}$ is a Sylow $\ell$-subgroup of $G$.
So we only need to consider the primes $\ell$ dividing $|G|$; otherwise $G_{\ell} = \{ 1 \}$ and $\ed_{\CC}(G_{\ell}; \ell) = 0$.

If $\ell \neq 2$ or $p$, then by Lemma~\ref{dih} (b), $G_{\ell}$ is cyclic; hence, $\ed_{\CC}(G_{\ell}) = 1$. 

If $\ell = p$, then by  Lemma~\ref{dih} (a), $G_{\ell} = G_p = (C_p)^r$, and $\ed_{\CC}(G_p; p)=r$.

If $\ell = 2$ and $p$ is odd, then by Lemma~\ref{dih} (c), $G_{\ell}$ is  a dihedral group; 
hence, $G_{\ell}$ has a 2-dimensional faithful linear representation
over $\mathbb C$. We conclude that $\ed_{\CC}(G_{2}; 2) \leqslant 2$.  On the other hand, since $G_2$ is not cyclic 
and $|G_2|\equiv 0\pmod{4}$, 
$\ed_{\CC} G_{\ell}\geq 2$ by~\cite[Theorem 6.2]{bur}.  So $\ed_{\mathbb{C}} G_{\ell}=2$.

This proves part (a) for the case that $p$ is odd. The case that $p$ is even follows directly 
from Lemma~\ref{dih} by the same method. 

(b) Note that the assertion of part (b) is vacuous if $\ell=p$ or $p = 2$.
So we may assume that $p$ is odd and $\ell\neq p$. Then it follows from Lemma~\ref{dih} that the Sylow $\ell$-subgroup of
$\PSL_2(q)$ is cyclic if $\ell$ is odd and dihedral if $\ell = 2$. Thus, by~\eqref{Gsize},  
 $\PSL_2(q)$ contains an element of order ${\ell}^s$, and the desired inequality follows from Proposition~\ref{prop2}.
\end{proof}

\begin{remark} \label{rem1.psl2} Note that, 
for odd $\ell$, Proposition~\ref{prop.psl2}(a) gives the ``Type I" lower bound:  
$\ed_{\CC}\PSL_2(q) \geqslant \max \{ 2, r \}$; cf. Remark~\ref{rem.semidirect2}.
We also know which finite simple groups have essential dimension $1$, $2$ or $3$ from \cite[Theorem 6.2]{bur}, \cite{duncan-ed2} and \cite{beauville},
respectively. Thus the lower bound of Proposition~\ref{prop.psl2}(b) is only of interest in those cases, where
\[ s \geqslant \max \{r + 1, 5 \}. \]
In such cases an unconditional proof of the lower bound \[ \ed_{\CC}(\PGL_2(q)) \geqslant s \]
(i.e., a proof that does not rely on Ledet's conjecture) is not known.
\end{remark}

\begin{remark} \label{rem2.psl2}
It follows from Proposition~\ref{prop.psl2}(a) that $\ed_{\CC}(\PSL_2(q)) \geqslant \edloc_{\CC}(\PSL_2(q) \geqslant r$
for any $q = p^r$. Hence, if we want 
to deduce an interesting (Type II) lower bound on $\ed_{\CC}(\PSL_2(q))$ from Proposition~\ref{prop2}, we need
$\ell^s$ to divide $q \pm 1 = p^r \pm 1$ for some prime $\ell$ and some integer $s \geqslant r + 1$.
This can only happen if $\ell < p$. In particular, this method gives no new information 
about $\ed_{\CC}(\PSL_2(q))$ in the case, where $q$ is a power of $2$.
\end{remark}

%
%
%
\begin{example} \label{ex2.psl}
Let $p = 31$ and $q = p^2 = 961$. Then $(q-1)/2 = 960$ is divisible by $2^6$. Thus Proposition~\ref{prop.psl2} yields

\smallskip
(a) $\edloc (\PSL_2(961)) =2$ but
(b) $\ed_{\CC}(\PSL_2(961)) \geqslant 5$.

\smallskip
Now let $q = p = 65537$. Note that $p$ is a Fermat prime, $p = 2^{16} + 1$. Here
Proposition~\ref{prop.psl2} yields 

\smallskip
(a) $\edloc (\PSL_2(65537)) =2$ but (b) $\ed_{\CC}(\PSL_2(65537)) \geqslant 15$. 

\smallskip
In both cases the inequality (b) is conditional on Ledet's conjecture.
\end{example}

\begin{remark} \label{rem3.psl2}
It follows from \cite[Theorem 2]{reichstein-cremona} that for any $d \geqslant 1$ there are only finitely many non-abelian simple finite groups $G$ such that $\ed_{\CC}(G) \leqslant d$. In some ways this assertion is more satisfying than the inequality of Proposition~\ref{prop.psl2}(b): it is unconditional (does not rely on Ledet's conjecture), and it covers all finite simple groups, 
not just those of the form $\PSL_2(q)$. On the other hand, it does not give an explicit lower bound on $\ed_{\CC}(G)$ for any particular 
finite simple group $G$.
\end{remark}

\section{Proof of Theorem~\ref{thm.main}}
\label{sect.proof}

We begin by remarking that an \'etale gerbe $\cX \arr S$ is weakly tame if and only if there exists an \'etale cover $\{S_i \arr S\}$ such that each $\cX_{S_i} \arr S_i$ is equivalent to $\cB_{S_i}G_i \arr S_i$
with $G_i$ weakly tame \'etale group schemes over $S_i$.

Our proof of Theorem~\ref{thm.main} will rely on the following Lemma~\ref{lem.versal}. To state it, we need the notion of \emph{versal object} of an algebraic stack. This is standard for classifying stacks of algebraic groups, but does not seem to be in the literature in the general case, so a short discussion is in order.

Let $\cX \arr \spec F$ be an algebraic stack of finite type over a field. Then $\cX$ preserves inductive limits, in the following sense: if $\{A_{i}\}$ is an inductive system of $F$-algebras over a filtered poset, the induced functor $\indlim \cX(A_{i}) \arr \cX(\indlim A_{i})$ is an equivalence of categories. If $L$ is an extension of $F$ then we can view $L$ as the inductive limit of its subalgebras $R \subseteq K$ of finite type over $F$; hence, given an object $\xi \in \cX(L)$, there exists a finitely generated subalgebra $R \subseteq K$ and an object $\xi_{R} \in \cX(R)$ whose image in $\cX(L)$ is isomorphic to $\xi$.

We say that an object $\xi \in \cX(L)$ is \emph{versal} if it satisfies the following condition, which expresses the fact that every object of $\cX$ over an extension of $F$ can be obtained by specialization of $\xi$.

For any $R$ and $\xi_{R}$ as above, and any object $\eta \in \cX(K)$ over an extension $K$ of $F$ that is an infinite field, there exists a homomorphism of $F$-algebras $R \arr K$ such that the image of $\xi_{R}$ in $\cX(K)$ under the induced functor $\cX(R) \arr \cX(K)$ is isomorphic to $\eta$.

Versal object don't exist in general; for example, they don't exist when $\cX$ has positive-dimensional moduli space. When they do exist, however, they control the essential dimension, that is, $\xi \in \cX(L)$ is versal, then the essential dimension of $\xi$ is easily seen to be the essential dimension of $\cX$ (in other words, no object of $\cX$ defined over a field can have essential dimension larger than that of $\xi$).

\begin{lemma}\label{lem.versal}
Let $\cX_{F} \arr \spec F$ be a finite \'etale gerbe over a field $F$. Suppose that $A$ is a non-zero finite $F$-algebra, and that the morphism $\spec A \arr \spec F$ has a lifting $\phi\colon \spec A \arr \cX_{F}$. Consider the locally free sheaf $\phi_{*}\cO_{\spec A}$ on $\cX_{F}$; call $\cV \arr \cX_{F}$ the corresponding vector bundle on $\cX_{F}$. Then $\cV$ has a non-empty open subscheme $U \subseteq \cV$. Furthermore, if $k(U)$ is the field of rational functions on $U$, the composite $\spec k(U) \arr U \subseteq \cV \arr \cX_{F}$ gives a versal object of $\cX_{F}\bigl(k(U)\bigr)$.
\end{lemma}

\begin{proof}
Let us show that $\cV$ is generically a scheme. We can extend the base field $F$, so that it is algebraically closed; in this case $\cX_{F}$ is the classifying space $\cB_{F}G$ of a finite group $G$, and there exists a homomorphism of $F$-algebras $A \arr F$. The vector bundle $\cV \arr \cX_{F}$ corresponds to a representation $V$ of $G$; by the semicontinuity of the degree of the stabilizer for finite group actions, it is enough to show that $V$ has a point with trivial stabilizer. The homomorphism $A \arr F$ gives a morphism $\spec F \arr \spec A$, and the composite $\spec F \arr \spec A \arr \cB_{F}G$ corresponds to the trivial $G$-torsor on $\spec F$. If we call $\cW$ the pushforward of $\cO_{\spec F}$ to $\cB_{F}G$, then $\cW \subseteq \cV$. On the other hand $\cW$ corresponds to the regular representation of $G$, and so the generic stabilizer is trivial, 
which proves what we want.

Let us show that the composite $\spec k(U) \arr U \subseteq \cV \arr \cX_{F}$ is versal; the argument is standard. Suppose that $K$ is an extension of $F$ that is an infinite field, and consider a morphism $\spec K \arr \cX_{F}$. It is enough to prove that for any open subscheme $U \subseteq \cV$, the morphism $\spec K \arr \cX_{F}$ factors through $U \subseteq \cV \arr \cX_{F}$. The pullback $V_{K} \arr \spec K$ of $\cV \arr \cX_{F}$ is a vector space on $K$, and the inverse image $U_{K} \subseteq V_{K}$ of $U \subseteq \cV$ is a non-empty open subscheme; hence $U_{K}(K) \neq \emptyset$, which ends the proof.
\end{proof}

\begin{proof}[Proof of Theorem~\ref{thm.main}] Let $\widehat{R}$ be the completion of $R$ and $\widehat{K}$ be the fraction field of 
	$\widehat{R}$. Then clearly $K \subset \widehat{K}$ and thus $\ed_{K}(\cX_K) \geqslant \ed_{\widehat{K}}(\cX_{\widehat{K}})$. Thus
for the purpose of proving Theorem~\ref{thm.main}, we may replace $R$ by $\widehat{R}$. In other words, we may (and will) assume that $R$ is complete.

Let $R \arr A$ be an \'etale faithfully flat algebra such that $\cX(A) \neq \emptyset$; since $R$ is henselian, by passing to a component of $\spec A$ we can assume that $R \arr A$ is finite. An object of $\cX(A)$ gives a lifting $\phi\colon \spec A \arr \cX$; this is flat and finite. Let $\cV \arr \cX$ be the vector bundle corresponding to $\phi_{*}\cO_{\spec A}$. If $U \arr \cV$ is the largest open subscheme of $\cV$, the Lemma above implies that $U \arr \spec R$ is surjective. Denote by $U_{K}$ and $U_{k}$ respectively the generic and special fiber of $U \arr \spec R$; call $E$ and $E_{0}$ the fields of rational functions on $U_{K}$ and $U_{k}$ respectively. Again because of the Lemma, the objects $\xi\colon \spec E \arr \cX_{K}$ and $\xi_{0}\colon \spec E_{0} \arr \cX_{k}$ are versal.

Consider the local ring $\cO_{E}$ of $U$ at the generic point of $U_{k}$, which is a DVR. The residue field of $\cO_{E}$ is $E_{0}$, and we have a morphism $\Xi\colon \spec\cO_{E} \arr \cX$ whose restrictions to $\spec K$ and $\spec k$ are isomorphic to $\xi$ and $\xi_{0}$ respectively.

Set $m \eqdef \ed_{K}\cX_{K}$; we need to show that $\xi_{0}$ has a compression of transcendence degree at most $m$.

There exists a field of definition $K \subseteq L \subseteq E$ for $\xi$ such that $\trdeg_{K}L = m$; call $\theta\colon \spec L \arr \cX$ the corresponding morphism, so that we have a factorization $\spec E \arr \spec L \xarr{\theta} \cX$ for $\xi$. Consider the intersection $\cO_{L}\eqdef \cO_{E} \cap L \subseteq E$; then $\cO_{L}$ is a DVR with quotient field $L$. Call $L_{0}$ it residue field; we have $L_{0} \subseteq E_{0}$. By Lemma~\ref{lem.valuation},
$\trdeg_{k}L_{0} \leqslant\trdeg_{K}L$.

Now it suffices to show that $\xi_{0}\colon \spec E_{0} \arr \cX$ factors through $\spec L_{0}$. Assume that we have proved that the morphism $\theta\colon \spec L \arr \cX$ extends to a morphism $\Theta\colon \spec \cO_{L} \arr \cX$. The composite $\spec E \subseteq \spec \cO_{E} \xarr{\Xi} \cX$ is isomorphic to the composite $\spec E \arr \spec L \subseteq \spec\cO_{L} \xarr{\Theta} \cX$; since $\cX$ is separated, it follows from the valuative criterion of separation that the composite $\spec \cO_{E} \arr \spec \cO_{L} \xarr{\Theta} \cX$ is isomorphic to $\Xi\colon \spec \cO_{E} \arr \cX$. By restricting to the central fibers we deduce that $\xi_{0}\colon \spec E_{0} \arr \cX$ is isomorphic to the composite $\spec E_{0} \arr \spec L_{0} \arr \cX$, and we are done.

To prove the existence of the extension $\Theta\colon \spec \cO_{L} \arr \cX$, notice that the uniqueness of such extension implies that to prove its existence we can pass to a finite \'etale extension $R \subseteq R'$, where $R'$ is a DVR; it is straightforward to check that formation of $\cO_{L}$ and $\cO_{E}$ commutes with such a base change. Hence we can assume that $\cX$ has a section, so that $\cX = \cB_{R}G$, where $G\arr \spec R$ is a finite \'etale weakly tame group scheme. By passing to a further covering we can assume that $G \arr \spec R$ is constant, that is, the product of $\spec R$ with a finite group $\Gamma$. If $A$ is an $R$-algebra, an action of $G$ on $\spec A$ corresponds to an action of $\Gamma$. 

The vector bundle $\cV \arr \cX$ corresponds to a vector bundle $V_{R} \arr \spec R$ with an $R$-linear action of $\Gamma$, such that the induced representations of $\Gamma$ on $V_{K}$ and $V_{k}$ are faithful. Call $\widetilde{E}$ the function field of $V_{K}$ and $\widetilde{E}_{0}$ the function field of $V_{k}$; then $\widetilde{E}^{\Gamma} = E$, and therefore $\cO_{\widetilde{E}}^{\Gamma} = \cO_{E}$. The factorization $\spec E\arr \spec L \arr \cX$ gives a $\Gamma$-torsor $\spec \widetilde{L} \arr \spec L$ whose lift to $\spec E$ is isomorphic to $\spec\widetilde{E} \arr \spec E$; then $\widetilde{L}$ is a $\Gamma$-invariant subfield of $\widetilde{E}$. Then $\cO_{\widetilde{L}} \eqdef \widetilde{L} \cap \cO_{\widetilde{E}}$ is a $\Gamma$-invariant DVR, and $\cO_{\widetilde{L}}^{\Gamma} = \widetilde{L}^{\Gamma} \cap \cO_{\widetilde{E}}^{\Gamma} = \cO_{L}$.

Call $\frm_{\widetilde{L}} \subseteq \cO_{\widetilde{L}}$ the maximal ideal, and set $\widetilde{L}_{0} \eqdef\cO_{\widetilde{L}}/\frm_{\widetilde{L}}$. If $t \in R$ is the uniformizing parameter, the image of $t$ in $\cO_{\widetilde{L}}$, which we denote again by $t$, is a uniformizing parameter; this is $\Gamma$-invariant. The action of $\Gamma$ on $\cO_{\widetilde{L}}$ descends to an action of $\Gamma$ on $\widetilde{L}_{0}$. By Proposition~\ref{prop.faithful}, this action is faithful.


So the action of $\Gamma$ on $\spec L_{0}$ is free over $k$; this implies that the action of $\Gamma$ on $\spec\cO_{\widetilde{L}} \arr \spec R$ is free, so $\spec\cO_{\widetilde{L}} \arr (\spec\cO_{\widetilde{L}})/G = \spec \cO_{L}$ is a $\Gamma$-torsor. This gives the desired morphism $\Theta\colon \spec\cO_{L} \arr \cX$, and ends the proof of the Theorem.
\end{proof}

\begin{remark}\label{remex}  
The problem with the proof of~\cite[Theorem 5.11]{brosnan-reichstein-vistoli3} was in the last sentence of the second paragraph
on page 1094.  We claimed there that 
the discrete valuation ring $R$ in the proof can be replaced
with the ring called $W(k(s))$.  
Since the essential dimension of the generic point can go up when we make 
this replacement, this is, in fact, not allowable.  (In effect, our mistake 
boils down to using an inequality in the wrong direction.)   

Note also that the proof of the characteristic $0$ genericity theorem 
in~\cite{brv1a} does not rely on Theorem~\ref{thmbrv3}. For that argument, 
which was different from the proof 
of~\cite[Theorem 6.1]{brosnan-reichstein-vistoli3}, 
see~\cite[Theorem 4.1]{brv1a}.
\end{remark}


\bibliographystyle{amsalpha}
\bibliography{ed-weakly-tame}
\end{document}